\documentclass[12pt]{article}
\usepackage{latexsym}
\usepackage{amsmath,amsfonts,amssymb,mathrsfs,graphicx}
\usepackage{amsthm}
\usepackage{verbatim}

\title{A capacity approach to box and packing dimensions of projections and other images}
\author{K.J. Falconer\\
\small{{\it Mathematical Institute,  
University of St~Andrews, North Haugh, St~Andrews,}} \\
\small{{\it Fife, KY16~9SS, Scotland }}} 
\date{}

\def\rn{\mathbb{R}^n}

\newcommand{\E}{\mathbb{E}}
\renewcommand{\P}{\mathbb{P}}

\newcommand{\fw}{f_\omega}

\newcommand\ubd{\overline{\mbox{\rm dim}}_{\rm B}} 
\newcommand\lbd{\underline{\mbox{\rm dim}}_{\rm B}} 
\newcommand\bdd{\mbox{\rm dim}_{\rm B}}
\newcommand\pkd{\mbox{\rm dim}_{\rm P}} 
\newcommand\hdd{\mbox{\rm dim}_{\rm H}} 

\newcommand{\be}{\begin{equation}} 
\newcommand{\ee}{\end{equation}} 

 \newtheorem{theo}{Theorem}[section]
 \newtheorem{cor}[theo]{Corollary}
 \newtheorem{lem}[theo]{Lemma}
 \newtheorem{prop}[theo]{Proposition}

\begin{document}
\maketitle

\begin{abstract}
\noindent Dimension profiles were introduced in \cite{FH2, How} to provide formulae for the box-counting and packing dimensions of the orthogonal projections of a set $E \subset \rn$ or a measure on $\rn$ onto almost all $m$-dimensional subspaces. The original definitions of dimension profiles were somewhat awkward and not easy to work with. Here we rework this theory with an alternative definition of dimension profiles in terms of capacities of $E$  with respect to certain kernels, and this leads to the box-counting dimensions of projections and other images of sets relatively easily. We also discuss other uses of the profiles, such as the information they give on exceptional sets of projections and dimensions of images under certain stochastic processes. We end by relating this approach to packing dimension.
 \end{abstract}

\section{Introduction}
\setcounter{equation}{0}
\setcounter{theo}{0}

The relationship between the Hausdorff dimension of a set $E \subset \rn$ and its orthogonal projections $\pi_V(E)$ onto  subspaces $V\in G(n,m)$, where $G(n,m)$ is the Grassmanian of $m$-dimensional subspaces of $\rn$ and $\pi_V: \rn \to V$ denotes orthogonal projection, has been studied since the foundational work of Marstrand \cite{Mar} for $G(2,1)$ extended to general $G(n,m)$ by Mattila \cite{Mat4}. They showed that for Borel $E \subset \rn$ 
\be\label{marmat}
\hdd \pi_V(E) \ = \ \min\{\hdd E, m\}  
\ee
for almost all $m$-dimensional subspaces $V$ with respect to the natural invariant probability measure $\gamma_{n,m}$ on $G(n,m)$, where $\hdd$ denotes Hausdorff dimension. Kaufman \cite{Kau,KM} gave a proof of these results using capacities and this has become the standard approach for  such problems.
Numerous generalisations, specialisations and consequences of these projection results have been developed, see \cite{FFJ,Mat5} for recent surveys.

It is natural to seek projection results for other notions of dimension. However,  examples show that the direct analogue of \eqref{marmat} is not valid for box-counting (Minkowski) dimension or packing dimension, though there are non-trivial lower bounds on the dimensions of the projections, see  \cite{FH,FM,Jar}. That the box-counting and packing  dimensions of the projections of a Borel set $E$ are constant for almost all subspaces $V \in G(n,m)$ was established in \cite{FH2,How} but this constant value, given by  a {\it dimension profile} of $E$ was specified  somewhat indirectly. For packing dimensions this is given in terms of the suprema of dimension profiles of measures supported by $E$ which in turn are given by critical parameters for certain almost sure pointwise limits  \cite{FH2}. The approach in \cite{How} defines box-counting dimension profiles in terms of weighted packings subject to constraints.  

These definitions of dimension profiles  are, frankly, messy, indirect and unappealing. In an attempt to make the concept more attractive, we present here an alternative approach to box-counting dimension profiles and their application to projections and other images in terms of capacities with respect to certain kernels. Using  simple properties of equilibrium measures leads to a direct and more natural formulation of  dimension profiles and the derivation of projection properties. 

Thus we will in \eqref{dimpro} define the  $s$-{\it box dimension profile} of $E\subset \mathbb{R}^n$ for $s> 0$ as
$$\bdd^s E\  =\   \lim_{r\to 0} \frac{\log  C_r^s(E)}{-\log r} $$
 where $C_r^s(E)$ is the capacity of $E$ with respect to a continuous kernel \eqref{ker}. (More exactly we will use upper and lower dimension profiles to correspond  to upper and lower limits should the limit not exist.) We will show 
in Section \ref{sec2.2} that if $s\geq n$ then $\bdd^s E$ is just the usual box-counting dimension of $E$, but in Section \ref{sec2.4} that if $1 \leq m \leq n-1$ then $\bdd^m E$ equals the box-counting dimension of the projection of $E$ onto almost every $m$-dimensional subspace of $\mathbb{R}^n$. In this way, the dimension profile $\bdd^s E$ may be thought of as the dimension of $E$ when regarded from an $s$-dimensional viewpoint. Analogously,  $\hdd^s E= \min\{\hdd E, s\}$ could be interpreted as  the Hausdorff dimension profile for Marstrand's result \eqref{marmat}. 

Since their introduction, dimension profiles have become a key tool in investigating the packing and box dimensions of the images of sets under random processes, see Section 2.5 and, for example, \cite{SX,Xi}.

\section{Capacities and box-counting dimensions}
\setcounter{equation}{0}
\setcounter{theo}{0}

Throughout this section we will consider images of a non-empty compact set  $E$; since box dimensions are not defined for unbounded sets or for the empty set, and also the box dimensions of a set equal those of its closure, we lose little by doing so. We will assume without always saying so explicitly that $E$ is non-empty. 

\subsection{Capacity and minimum energy}
Capacity arguments using potential kernels of the form $\phi(x) = |x|^{-s}$ are widely used in Hausdorff dimension arguments, see for example \cite{Kau,KM,Mat,Mat2}. For box-counting dimensions, another class of kernels turns out to be useful.
Let $s>0$ and $r>0$. Throughout this paper we use the potential kernels 
\be\label{ker}
 \phi_r^s(x)= \min\Big\{ 1, \Big(\frac{r}{|x|}\Big)^s\Big\} \qquad (x\in \rn)
\ee
which were introduced in \cite{FH,FM}.
Let $E\subset \rn$ be non-empty and compact and let   ${\mathcal M}(E)$ denote the set of Borel probability measures supported by $E$. We define the {\em energy} of  $\mu  \in {\mathcal M}(E)$ with respect to the kernel $\phi_r^s$ by
$$\int\int \phi_r^s(x-y)d\mu(x)d\mu(y),$$
and the {\em potential}  of $\mu$  at $x\in \rn$ by 
$$ \int \phi_r^s(x-y)d\mu (y).$$
The {\em capacity} $C_r^s(E) $ of $E$  is the reciprocal of the minimum energy achieved by probability measures on $E$, that is 
\be\label{minen}
\frac{1}{C_r^s(E)}\  = \ \inf_{\mu \in {\mathcal M}(E)}\int\int \phi_r^s(x-y)d\mu(x)d\mu(y);
\ee
note that since  $\phi_r^s$ is continuous and $E$ is  compact, $0<C_r^s(E)<\infty$. (For a non-closed bounded set the capacity is taken to equal that of its closure.)

The following energy-minimising property is standard in potential theory, but it is key for our needs, so we give the proof which is particularly simple for continuous kernels.

\begin{lem}\label{equilib}
Let $E\subset \rn$ be compact and $s>0$ and $r>0$.  Then the infimum in \eqref{minen} is attained by a measure $\mu_0  \in {\mathcal M}(E)$. Moreover 
\be\label{pot}
 \int \phi_r^s(x-y)d\mu_0  (y)\ \geq\  \frac{1}{C_r^s(E)}
\ee
for all $x\in E$, with equality for $\mu_0 $-almost all $x \in E$.
\end{lem}
\begin{proof}
Let $\mu_k \in {\mathcal M}(E)$ be such that $\int\int \phi_r^s(x-y)d\mu_k(x)d\mu_k(y) \to \gamma := 1/C_r^s(E)$. Then $\mu_k$ has a  subsequence that is weakly convergent to some $\mu_0  \in {\mathcal M}(E)$; since $\phi_r^s(x-y)$ is continuous the infimum is attained.

Now suppose that $\int \phi_r^s(w-y)d\mu_0(y)\leq \gamma - \epsilon$ for some $w\in E$ and $\epsilon >0$. Let $\delta_w$ be the unit point mass at $w$ and for $0<\lambda<1$ let $\mu_\lambda =\lambda \delta_w + (1-\lambda)\mu_0  
\in {\mathcal M}(E)$. Then
\begin{eqnarray*}
\int\int \phi_r^s(x-y)d\mu_\lambda(x)d\mu_\lambda(y)
&=& \lambda^2 \phi_r^s(w-w) +2\lambda(1-\lambda)\int \phi_r^s(w-y)d\mu_0 (y)\\
&&+\ (1-\lambda)^2 \int\int \phi_r^s(x-y)d\mu_0 (x)d\mu_0 (y)\\
&\leq& \lambda^2 +2\lambda(1-\lambda)(\gamma - \epsilon) +(1-\lambda)^2\gamma\\
&=& \gamma - 2\lambda\epsilon + O( \lambda^2),
\end{eqnarray*}
which, on taking $\lambda$ small, contradicts that $\mu_0 $ minimises the energy integral.
Thus  inequality \eqref{pot} is satisfied for all $x\in E$, and equality for $\mu_0 $-almost all $x$ is immediate from \eqref{minen}. 
\end{proof}

\subsection{Capacities and box-counting numbers}\label{sec2.2}

For a non-empty compact  $E\subset \rn$,  let $N_r(E)$ be the minimum number of sets of diameter $r$ that can cover $E$. Recall that the {\em lower} and {\em upper box-counting dimensions} of $E$ are defined by 
\be\label{boxdims}
\lbd E\ =\ \varliminf_{r\to 0} \frac{\log  N_r(E)}{-\log r} 
\quad \mbox{ and }\quad  \ubd E\ = \ \varlimsup_{r\to 0} \frac{\log  N_r(E)}{-\log r}, 
\ee
with the {\em box-counting dimension} given by the common value if the limit exists, see for example \cite{Fa}.
The aim of this section is to prove Corollary  \ref{capcor},  that the capacity $C_r^s(E)$ and the covering number $N_r(E)$ are comparable provided that $s\geq n$. Note that this is not necessarily the case if $0\leq s<n$ and we will see in Subsection \ref{projdims} that it is this disparity that leads to formulae for the box dimensions of projections. The next two lemmas obtain lower and upper bounds for $N_r(E)$.

\begin{lem}\label{genbound}
Let $E\subset \rn$ be compact and let $r>0$. 
Suppose that $E$ supports a measure  $\mu\in {\mathcal M}(E)$ such that for some $\gamma>0$
\be\label{est0}
(\mu \times \mu)\big\{(x,y): |x-y| \leq r\big\}\ \leq\ \gamma.
\ee
Then 
\begin{equation}\label{concl1}
N_{r}(E)\  \geq\  \frac{c_n}{\gamma},
\end{equation}
where the constant $c_n$ depends only on $n$. In particular \eqref{concl1} holds if, for some $s>0$,
\be\label{est0a}
\int\int \phi_r^s(x-y)d\mu(x)d\mu(y)\ \leq\ \gamma.
\ee
\end{lem}

\begin{proof}
Let ${\cal C}(E)$ be the set of half-open coordinate mesh cubes of diameter $r$ (i.e. of side length $r/\sqrt{n}$) that intersect $E$, and suppose that there are  $N'_{r}(E)$ such cubes.
By  Cauchy's inequality,
\begin{eqnarray*}
1\ =\ \mu(E)^2    &= & \bigg(\sum_{C\in {\cal C}(E)}\mu(C)\bigg)^2\\
& \leq &  N'_{r}(E) \sum_{C\in {\cal C}(E)}\mu(C) ^2\\
& = &  N'_{r}(E) \sum_{C\in {\cal C}(E)}(\mu \times \mu)\big\{(w, z)\in C\times C\big\}\\
& \leq &  N'_{r}(E)\, (\mu \times \mu) \big\{(w, z): |w-z| \leq r\big\}\\
& \leq &  N'_{r}(E)\,  \gamma\\
& \leq &  (2 \sqrt{n})^{n}\, N_{r}(E)\,  \gamma,
\end{eqnarray*}
noting that a set of diameter $r$ can intersect at most $(2 \sqrt{n})^{n}$ of the cubes of ${\cal C}(E)$.

Since $1_{B(0,r)}(x)\leq \phi_r^s(x)$,  inequality \eqref{est0a}  implies \eqref{est0}, to complete the proof.
\end{proof}

\begin{lem}\label{potbound}
Let $E\subset \rn$ be non-empty and compact and let $s>0$ and $r>0$. 
Suppose that $E$ supports a measure  $\mu\in {\mathcal M}(E)$ such that for some $\gamma>0$
\be\label{est2}
\int \phi_r^s(x-y)d{\mu} (y)\ \geq \ \gamma \qquad \mbox{ for all } x\in E.
\ee
Then
\be\label{est2c}
N_r(E)\  \leq \  
\left\{
\begin{array}{ll}
{\displaystyle \frac{c_{n,n}\lceil\log_2 ({\rm diam}E / r)+1\rceil}{\gamma } }& \mbox{ if } s=n    \\
{\displaystyle  \frac{c_{n,s}}{\gamma}}&  \mbox{ if } s>n   
\end{array}
\right.  ,
\ee
where $c_{n,s}$ depends only on $n$ and $s$.
\end{lem}

\begin{proof}
 Write $M = {\rm diam}E$. For all $x\in E$,
\begin{eqnarray*}
 \int \phi_r^s(x-y)d{\mu} (y)
&\leq& \mu(B(x,r)) +\sum_{k=0}^{\lceil\log_2 (M/r)-1\rceil} \int _{B(x,2^{k+1}r)\setminus  B(x,2^{k}r)}2^{-ks} d\mu(y)\\
&\leq& \mu(B(x,r)) +\sum_{k=0}^{\lceil\log_2 (M/r)-1\rceil} 2^{-ks} \mu(B(x,2^{k+1}r)) \\
&\leq& 2^s\sum_{k=0}^{\lceil\log_2 (M/r)\rceil} 2^{-ks} \mu(B(x,2^{k}r)).
\end{eqnarray*}
Let $B(x_i, r),\  i = 1,\ldots, N'_r(E)$, be a maximal collection of disjoint balls of radii $r$ with $x_i \in E$, where  in this proof $N'_r(E)$ denotes this maximum number.
From \eqref{est2}, for each $i$,
$$ \gamma\ \leq\   \int \phi_r^s(x_i-y)d{\mu} (y)\ \leq2^s\ \sum_{k=0}^{\lceil\log_2 (M/r)\rceil} 2^{-ks} \mu(B(x_i,2^{k}r)).$$
Summing over the $x_i$,
$$N'_r(E)\gamma\  \leq   \sum_{k=0}^{\lceil\log_2 (M/r)\rceil} 2^{s(1-k)} \sum_{i=1}^{N'_r(E)}  \mu(B(x_i,2^{k}r)),$$
so, for some $k$ with $0\leq k \leq\lceil \log_2 (M/r)\rceil$,
\be\label{est3}
2^{s(1-k)}\sum_{i=1}^{N'_r(E)}  \mu(B(x_i,2^{k}r))  \ \geq\
\left\{
\begin{array}{ll}
 N'_r(E)\gamma\big/  \lceil\log_2 (M/r)+1\rceil   & \mbox{ if } s=n    \\
 N'_r(E)\gamma\, 2^{k(n-s)}  (1-2^{n-s}) &  \mbox{ if } s>n   
\end{array}
\right. ,
\ee
the case of $s>n$ coming from comparison with a geometric series. For all $x\in E$ a volume comparison using the disjointedness of the balls $B(x_i,r)$ shows that at most $ (2^k +1)^n\leq 2^{(k+1)n}$ of the $x_i$ lie in $B(x,2^{k}r)$. Consequently $x$ belongs to at most $2^{(k+1)n}$ of the $B(x_i,2^{k}r)$. Thus 
\be\label{est4}
\sum_{i=1}^{N'_r(E)}  \mu(B(x_i,2^{k}r))\ \leq\ 2^{(k+1)n}\mu(E)\ 
=\ 2^{n+s}2^{-s(1-k)}2^{k(n-s)}\ \leq \ 2^{n+s} 2^{-s(1-k)},
\ee
using that $s\geq n$.
Inequality  \eqref{est2c} now  follows from \eqref{est3},  \eqref{est4} and the fact that $N_r(E)\leq a_n N'_r(E)$ where $a_n$ is the minimum number of balls  in $\mathbb{R}^n$ of diameter $1$ that can cover a ball of radius 1.
\end{proof}

\begin{cor}\label{capcor}
Let $E\subset \rn$ be non-empty and compact and let $r>0$. Then
\be\label{capineq}
c_n C^s_r(E)\  \leq\ N_r(E)\  \leq \  
\left\{
\begin{array}{ll}
{\displaystyle c_{n,n}\lceil\log_2 ({\rm diam}E / r)+1\rceil\  C^s_r(E) }& \mbox{ if } s=n    \\
{\displaystyle  c_{n,s}\ C^s_r(E)}&  \mbox{ if } s>n   
\end{array}
\right.  .
\ee
\end{cor}

\begin{proof}
By Lemma \ref{equilib}  we may find $\mu\in {\mathcal M}(E)$ satisfying \eqref{minen} and \eqref{pot}, so the conclusion follows immediately from Lemmas \ref{genbound} and \ref{potbound}. 
\end{proof}

\subsection{Dimension profiles and box-counting dimensions of images}\label{projdims}
For $s> 0$ we define the {\it lower} and {\it upper}  $s$-{\it box dimension profiles} of $E\subseteq \mathbb{R}^n$ by
\be\label{dimpro}
\lbd^s E\  =\   \varliminf_{r\to 0} \frac{\log  C_r^s(E)}{-\log r} \quad \mbox{and}\quad  \ubd^s E\   =\   \varlimsup_{r\to 0} \frac{\log  C_r^s(E)}{-\log r}.
\ee
When $s\geq n$  equality of the box dimensions and the dimension profiles is immediate from  Corollary \ref{capcor} . 

\begin{cor}\label{sgeqn}
Let $E\subset \mathbb{R}^n$. If $s\geq n$ then 
$$\lbd^s E\   =\   \lbd E  \quad \mbox{and}\quad  \ubd^s E\   = \  \ubd E.$$
\end{cor}
\begin{proof}
This follows from \eqref{capineq}   and the definitions of box dimensions \eqref{boxdims} and of dimension profiles \eqref{dimpro}.
\end{proof}

The following theorem enables us to obtain upper bounds for the box dimensions of images of sets under Lipschitz or 
H\"{o}lder functions in terms of dimension profiles.

\begin{theo}\label{liplem}
Let $E\subset \mathbb{R}^n$ be compact and let $f:E\to \mathbb{R}^m$ be an $\alpha$-H\"{o}lder map satisfying
\be\label{lip}
|f(x)-f(y)| \ \leq\  c |x-y|^\alpha \qquad (x,y \in E), 
\ee
where $c>0$ and $0<\alpha \leq 1$. Then 
$$\lbd f( E)\  \leq \  \frac{1}{\alpha}\, \lbd^{m\alpha} E  \quad \mbox{and}\quad  \ubd f( E)\  \leq \   \frac{1}{\alpha}\, \ubd^{m\alpha}  E.$$
\end{theo}

\begin{proof}
From \eqref{ker} and \eqref{lip}, for $r>0$ and $s>0$,
\begin{align*}
 \phi_r^s(f(x)-&f(y)) =  \min\Big\{ 1, \Big(\frac{r}{|f(x)-f(y)|}\Big)^s\Big\}
\ \geq\  \min\Big\{ 1, \Big(\frac{r}{c|x-y|^\alpha}\Big)^s\Big\}\\
& =\  \min\Big\{ 1, c^{-s}\Big(\frac{r^{1/\alpha}}{|x-y|}\Big)^{s\alpha}\Big\}\ \geq\ c_0\phi_{r^{1/\alpha}}^{s\alpha}(x-y)
\end{align*}
 for $x,y\in E$, where $c_0= \min\{1,c^{-s}\}$.

For each $r$ we may, by Lemma \ref{equilib}, find a measure $\mu  \in {\mathcal M}(E)$ such that for all $x\in E$
$$\frac{1}{C^{m\alpha}_{r^{1/\alpha}}(E)} \leq \int \phi_{r^{1/\alpha}}^{m\alpha}(x-y)d\mu  (y)\ 
 \leq \ c_0^{-1} \int \phi_{r}^{m}(f(x)-f(y))d\mu (y)\ 
\leq\  c_0^{-1}\int \phi_r^m(f(x)-w)d(f\!\mu) (w),$$
where  $f\!\mu  \in {\mathcal M}(f(E))$ is  the image of the measure $\mu$ under $f$ defined by $\int g(w) d(f\!\mu) (w) =  \int g(f(x)) d\mu(x)$ for continuous $g$. Then for each $z=f(x) \in f(E)$,
$$ \int \phi_r^m(z-w)d(f\!\mu) (w)\ \geq\  \frac{c_0}{C^{m\alpha}_{r^{1/\alpha}}(E)}.$$ 
By Lemma \ref{potbound} 
$$N_r(f(E))\ \leq\   c_{m,m} \lceil\log_2({\rm diam}f(E) / r)+1\rceil c_0^{-1}C^{m\alpha}_{r^{1/\alpha}}(E),$$
so 
$$\frac{\log  N_r(f(E))}{-\log r}
\   \leq\  \frac{\log\big(c_{m,m} \lceil\log_2 ({\rm diam}f(E) / r +1\rceil c_0^{-1}\big)}{-\log r}
\ + \ \frac{\log C^{m\alpha}_{r^{1/\alpha}}(E)}{-\alpha \log r^{1/\alpha}}$$
and the conclusion follows  on taking lower and upper limits as $r\searrow 0$.
\end{proof}

The next theorem enables us to obtain almost sure lower bounds for box dimensions of images of sets. It is convenient to express the condition in probabilistic terms, though for our principal application to projections $f_\omega$ will simply be orthogonal projection from $\mathbb{R}^n$ onto the $m$-dimensional subspace $\omega \in G(n,m)$.

In Theorem \ref{mainnew}, $(\omega, {\mathcal F}, \P)$ is a probability space and $(\omega, x) \mapsto f_\omega(x)$ is a $\sigma({\mathcal F}\times {\mathcal B})$-measurable function where ${\mathcal B}$ denotes the Borel sets in $\mathbb{R}^n$.

\begin{theo}\label{mainnew}
Let $E\subset \mathbb{R}^n$ and let $s, \gamma>0$. Let $\{f_\omega: E \to \mathbb{R}^m,\ \omega \in \Omega\}$ be a family of continuous mappings such that for some $c>0$
\be\label{probbound}
\P(|\fw(x) - \fw(y)|\leq r)\ \leq\ c\,\phi_{r^\gamma}^s (x-y)
\qquad (x,y \in E, \ r>0).
\ee
Then, for $\P$-almost all $\omega\in \Omega$, 
\be\label{plb}
\lbd \fw( E)\  \geq \  \gamma\, \lbd^{s} E  \quad \mbox{and}\quad  \ubd \fw( E)\  \geq \   \gamma\, \ubd^{s}  E.
\ee
\end{theo}

\begin{proof}
First note that for all $\mu  \in {\mathcal M}(E)$,\begin{align}
\E\big(  (\fw\mu \times \fw\mu)&\big\{(w, z):  |w-z| \leq r\big\}\big) \nonumber\\
&=\ \E\big( (\mu \times \mu)\big\{(x, y):  |\fw x -\fw y| \leq r\big\}\big)\nonumber\\
&=\ \int \int\P\big\{ |\fw x -\fw y| \leq r\big\}d\mu(x)d\mu(y)\nonumber\\
&\leq \ c \int\int \phi_{r^\gamma}^s (x-y)d\mu(x)d\mu(y)\label{pipi}
\end{align}
using \eqref{probbound}.
If  $ \ubd^s E >t'>t>0 $ then $C_{r_i^\gamma}^s(E) \geq {r_i}^{-\gamma t'}$ for a sequence $r_i \searrow 0$, where we may assume that $0< r_i \leq 2^{-i}$ for all $i$. Then for each $i$ there is a measure $\mu_i  \in {\mathcal M}(E)$ such that 
$$\int\int \phi_{r_i^\gamma}^s(x-y)d\mu_i(x)d\mu_i(y)\leq {r_i}^{\gamma t'}.$$
Summing, and using \eqref{pipi},
\begin{align*}
\E\Big(\sum_{i=1}^\infty r_i^{-\gamma t}  & (\fw\mu_i \times \fw\mu_i)\big\{(w, z):  |w-z| \leq r_i\big\}\Big)\\
   &\leq\ c\sum_{i=1}^\infty r_i^{-\gamma t}\int\int \phi_{r_i^\gamma}^s (x-y)d\mu_i(x)d\mu_i(y)\\
  &\leq\ c\sum_{i=1}^\infty r_i^{\gamma (t'-t)} \ \leq \ c\sum_{i=1}^\infty 2^{-i\gamma(t'-t)}\ <\  \infty.
\end{align*}
Thus, for almost all $\omega$  there is a number $M_\omega< \infty$ such that 
$$ (\fw\mu_i \times \fw\mu_i)\big\{(w, z):  |w-z| \leq r_i\big\}\leq M_\omega r_i^{\gamma t}$$
for all $i$.
For such $\omega$ the image measure $\fw\mu_i$ is supported by $\fw (E)\subset \mathbb{R}^m$, so by Lemma \ref{genbound}, 
$$N_{r_i} (\fw E )\ \geq\  c_m M_\omega^{-1}  r_i^{-\gamma t}$$
 for all $i$. 
Hence $\varlimsup_{r\to 0} \log (N_{r}(\fw E)/-\log r )\geq \gamma t$. This is so for all $t< \ubd^s E$, so  
$\ubd \fw E \  \geq\  \gamma  \ubd^s E$ for almost all $\omega$.

The inequality of the lower dimensions for almost all $\omega$ follows in a similar manner, noting that for estimating the dimensions it is enough to take  $r =2^{-i}$ for $ i\in \mathbb{N}$ when taking the limits as $r\searrow 0$ in the definition of lower box dimension.
\end{proof}

\subsection{Box-counting dimensions of projections}\label{sec2.4}

The most basic application of the theorems of Section \ref{projdims} is to orthogonal projection of sets. Recall that $\gamma_{n,m}$ is the normalised invariant measure on $G(n,m)$, the Grassmanian of $m$-dimensional subspaces of $\mathbb{R}^n$.

\begin{theo}\label{main}
Let $E\subset \mathbb{R}^n$ be compact. Then  for all $V\in G(n,m)$
\be\label{dbin}
\lbd \pi_V E \  \leq\   \lbd^m E  \quad \mbox{and}\quad  \ubd \pi_V E \  \leq\   \ubd^m E,
\ee
and for $\gamma_{n,m}$-almost all $V\in G(n,m)$
\be\label{dbeq}
\lbd \pi_V E \  =\   \lbd^m E  \quad \mbox{and}\quad  \ubd \pi_V E \  =\   \ubd^m E.
\ee
\end{theo}

\begin{proof}
Identifying $V$ with $\mathbb{R}^m$ in the natural way, $\pi_V: \mathbb{R}^n \to V$ is Lipschitz so the inequalities \eqref{dbin} follow from Theorem \ref{liplem} taking $\alpha=1$. 
 
 Now note that values  $\phi^m_r(x)$ are comparable to the proportion of the subspaces $V\in G(n,m)$ for which the  $r$-neighbourhoods of  the orthogonal subspaces to $V$ contain $x$, specifically, for all $1\leq m < n$ there are numbers $a_{n,m}>0$ such that 
\be\label{equiveqn}
\phi_r^m (x) \ \leq\ \gamma_{n,m}\big\{V: |\pi_Vx| \leq r\big\}\ \leq\ a_{n,m}\,\phi_r^m (x)
\qquad (x\in \rn).
\ee
This standard geometrical estimate can be obtained in many ways, see for example \cite[Lemma 3.11]{Mat}. One approach is to normalise to the case where $|x|$ = 1 and then estimate the (normalised) $(n-1)$-dimensional spherical area of $S \cap \{y: {\rm dist}(y,V^\perp)\leq r\}$, that is  the intersection of the unit sphere $S$ in $\rn$ with the `tube' or `slab' of points within distance $r$ of some $(n-m)$-dimensional subspace $V^\perp$ of $\rn$. In particular
\be\label{equiveqn1}
\gamma_{n,m}\big\{V: |\pi_Vx-\pi_Vy| \leq r\big\}\ \leq\ a_{n,m}\,\phi_r^m (x-y)
\qquad (x,y\in \rn).
\ee
Taking $\Omega = G(n,m)$ and $\P =\gamma_{n,m}$, this is \eqref{probbound} with $s=m$ and $\gamma = 1$. Thus \eqref{dbeq} follows from Theorem \ref{mainnew}.
\end{proof}

Whilst equality holds in  \eqref{dbeq} for $\gamma_{n,m}$-almost all $V \in G(n,m)$, dimension profiles can provide further information on the size of the set of $V $ for which the box dimensions of the projections $ \pi_V E$ are exceptionally small; this is analogous to the bounds on the dimensions of exceptional projections that have been obtained for Hausdorff dimensions \cite{KM,Mat4, Mat, Mat2}. Note that $ G(n,m)$ is a manifold of dimension  $m(n-m)$ and thus $\hdd G(n,m) =   m(n-m)$, where $\hdd$ denotes Hausdorff dimension.  In particular the dimension bound for the exceptional set of $V$ given by  the following theorem decreases as the deficit $m-s\geq 0$ increases. The proof is similar to that of Theorem \ref{main} except integration is with respect to a measure $\nu$ supported  on the exceptional set of $V$, with  \eqref{equiveqn} replaced by an estimate involving $\nu$ rather than $\gamma_{n,m}$.  
\begin{theo}
Let $E\subset \mathbb{R}^n$ be compact and let $0\leq s\leq m$. Then  
\be\label{dbeqa}
\hdd \{V \in  G(n,m) \mbox{ \rm such that } \ubd \pi_V E \,  < \,   \ubd^s E\}
\ \leq\  m(n-m) - (m-s),
\ee
and
\be\label{dbeqb}
\hdd \{V \in  G(n,m) \mbox{ \rm such that } \lbd \pi_V E \,  < \,   \lbd^s E\}
\ \leq\  m(n-m) - (m-s).
\ee
\end{theo}
\begin{proof}
Let
$$K = \{V \in  G(n,m) \mbox{ such that } \ubd \pi_V E \,  < \,   \ubd^s E\}.$$
Suppose for a contradiction that $\hdd K  >  m(n-m) - (m-s)$.   By Frostman's Lemma, see \cite{Mat4,Mat2}, there is a Borel probability measure $\nu$ supported on a compact subset of $K$ and $c>0$ such that 
$\nu (B_G(V,\rho)) \leq c\rho^{m(n-m) - (m-s)}$  for all $V \in G(n,m)$ and $\rho>0$, where   $B_G(V,\rho)\subset G(n,m)$ is the ball centre $V$ and radius $\rho$ with respect to some natural locally $m(n-m)$-dimensional metric on the Grassmanian manifold $G(n,m)$. This ensures that the subspaces in $K$ cannot be too densely concentrated, and analogously to  \eqref{equiveqn} and \eqref{equiveqn1} we obtain 
$$\nu\big\{V: |\pi_Vx-\pi_Vy| \leq r\big\}\ \leq\ a\,\phi_r^s (x-y)
\qquad (x\in \rn)$$
for some $a>0$, see \cite{Mat4} or \cite[Inequality\! (5.12)]{Mat2} for more details. Taking 
$\Omega = G(n,m)$ and $\P =\nu$ with  $\gamma = 1$ in  Theorem \ref{mainnew} gives that   $\ubd \pi_V E   \geq  \ubd^s E$ for $\nu$-almost all  $V\in  G(n,m)$, contradicting the definition of $k$. The proof for lower box dimensions is similar. 
\end{proof}

\subsection{Images under stochastic processes}
Almost immediately after their introduction, Xiao \cite{SX,Xi} used dimension profiles to determine the packing dimensions of the images of sets under fractional Brownian motions. In our framework, the box dimension analogues follow easily from Theorems \ref{liplem} and  \ref{mainnew}. We first state a result that applies to general random functions.

\begin{prop}\label{holder}
Let $X:\mathbb{R}^n \to \mathbb{R}^m$ be a random function, let $E \subset \mathbb{R}^n$ be compact and let $0<\alpha \leq 1$.  Suppose that 

{\rm (a)} for all $0<\epsilon <\alpha$ there is a random constant $M>0$ such that 
\be\label{holstoc}
|X(x)-X(y)| \ \leq\  M |x-y|^{\alpha-\epsilon} \qquad (x,y \in E), 
\ee
almost surely,  and

{\rm (b)}
for all $\epsilon>0$ there is a constant $c>0$
\be\label{probstoc}
\P\big(|X(x) - X(y)|\leq r\big)\ \leq\ c\,\Big(\frac{r^{1- \epsilon}}{|x-y|^{\alpha + \epsilon}}\Big)^m
\qquad (x,y \in E, \ r>0). 
\ee

\noindent Then, almost surely, 
$$\lbd X( E)\  = \  \frac{1}{\alpha}\, \lbd^{m\alpha} E  \quad \mbox{and}\quad  \ubd X( E)\  = \   \frac{1}{\alpha}\, \ubd^{m\alpha}  E.$$
\end{prop}
\begin{proof}
Note that condition \eqref{probstoc} implies that 
$$\P\big(|X(x) - X(y)|\leq r\big)\ \leq\ \min \Big\{1, c\,\Big(\frac{r^{(1-\epsilon)/(\alpha + \epsilon)}}{|x-y|}\Big)^{(\alpha + \epsilon) m}\Big\} 
\ \leq \ c_0 \phi^{(\alpha + \epsilon)m}_{r^{(1-\epsilon)/(\alpha + \epsilon)}}(x-y).$$
The conclusion is immediate using Theorem \ref{liplem} and Theorem \ref{mainnew} taking $\epsilon$ arbitrarily small.
\end{proof}

Index-$\alpha$ fractional Brownian motion $(0<\alpha <1)$ is the Gaussian random function $X:\mathbb{R}^n \to \mathbb{R}^m$  that with probability 1 is continuous with $X(0) = 0$ and such that the increments 
$X(x)-X(y)$ are multivariate normal with mean 0 and variance $|x-y|^\alpha$, see, for example,  \cite{Fa,Kah}. In particular, 
$X= (X_1, \ldots,X_m)$ where the $X_i:\mathbb{R}^n \to \mathbb{R}$ are independent index-$\alpha$ fractional Brownian motions with distributions given by 
\be\label{fracgauss}
\P\big(X_i (x)-X_i (y)\in A\big)\ \ =\ \frac{1}{\sqrt{2\pi}}\frac{1}{|x-y|^{\alpha}}
\int_{t\in A}  \exp\Big(\frac{-t^2}{2|x-y|^{2\alpha}}\Big) dt 
\ee
for each Borel set $A\subset \mathbb{R}$.

\begin{cor}
Let $X:\mathbb{R}^n \to \mathbb{R}^m$ be index-$\alpha$ fractional Brownian motion $(0<\alpha <1)$ and let $E \subset \mathbb{R}^n$ be compact. 
 Then, almost surely, 
$$\lbd X( E)\  = \  \frac{1}{\alpha}\, \lbd^{m\alpha} E  \quad \mbox{and}\quad  \ubd X( E)\  = \   \frac{1}{\alpha}\, \ubd^{m\alpha}  E.$$
\end{cor}
\begin{proof}
Index-$\alpha$ fractional Brownian motion satisfies an $(\alpha - \epsilon)$-H\"{o}lder condition for all $0<\epsilon <\alpha$, so \eqref{holstoc} is satisfied.

Furthermore, for each $\epsilon >0$
\begin{align*}
\P\big(|X(x) - X(y)|\leq r\big)\ \ & \leq\ \P\big((|X_i(x) - X_i(y)|\leq r\mbox{ for all } 1\leq i \leq m\big)\\ 
& \leq\ \bigg(\frac{1}{\sqrt{2\pi}}\frac{1}{|x-y|^{\alpha}}
\int_{|t|\leq r}  \exp\Big(\frac{-t^2}{2|x-y|^{2\alpha}}\Big) dt\bigg)^{m}  \\
&\leq c\bigg(\frac{r^{1-\epsilon} }{|x-y|^{\alpha(1+\epsilon)}}\bigg)^m
\end{align*}
using that $\exp(-u) \leq c_\epsilon u^{-\epsilon}$ for $u>0$, giving \eqref{probstoc}. The conclusion follows by Proposition \ref{holder}.
\end{proof}

\subsection{Inequalities}

We exhibit some inequalities satisfied by the dimension profiles; these were obtained in \cite[Section 6]{FH2} but their derivation is  more direct using our capacity approach.
\begin{prop}\label{ineqs}
Let $E\subset \mathbb{R}^n$ and set either  $d(s)= \lbd^s E$  or $d(s)= \ubd^s E$. Then  for  $0< s\leq t$,
\be\label{ineq1}
0\ \leq\ d(s) \ \leq\ d(t)\  \leq \ n, 
\ee
\be\label{ineq15}
 \qquad d(s) \ \leq s,
\ee
and 
\be\label{ineq2}
0\ \leq \ \frac{1}{d(s)} - \frac{1}{s} \ \leq\  \frac{1}{d(t)} - \frac{1}{t}.
\ee
\end{prop}

\begin{proof}
Inequality \eqref{ineq1} is immediate from comparison of the  kernels \eqref{ker} for $s$ and $t$. The bound \eqref{ineq15} follows by noting that $C_r^s(E)^{-1} = \int\phi_r^s(x-y)d\mu_0(y) \geq r^s\int_r^\infty|x-y|^{-s}d\mu_0(y)$ for some $x\in E$, where  $\mu_0$ is an energy-minimising measure on $E$, and this last integral is bounded away from 0 for small $r$.

For \eqref{ineq2}  let $0<r<R$, $0<s<t$ and $d>0$. Then for $\mu\in {\mathcal M}(E)$ and $x\in E$, splitting the integral and using H\"{o}lder's inequality,
\begin{eqnarray*}
\int \phi_r^s(x-y) d\mu(y) 
&\leq& \mu(B(x,R)) + \int_{|x-y|>R} \Big(\frac{r}{|x-y|}\Big)^sd\mu(y) \\
&=& \mu(B(x,R)) + r^sR^{-s}\int_{|x-y|>R} \Big(\frac{R}{|x-y|}\Big)^s d\mu(y) \\
&\leq & \mu(B(x,R)) +  r^sR^{-s}\bigg(\int_{|x-y|>R} \Big(\frac{R}{|x-y|}\Big)^t d\mu(y)\bigg)^{s/t} \\
&\leq & \int\phi_R^t(x-y) d\mu(y) +  r^sR^{-s}\bigg(\int\phi_R^t(x-y) d\mu(y)\bigg)^{s/t} \\
&\leq &R^{d} \bigg(R^{-d}\int\phi_R^t(x-y) d\mu(y)\bigg) +  r^sR^{s(d/t-1)}\bigg(R^{-d}\int\phi_R^t(x-y) d\mu(y)\bigg)^{s/t}
\end{eqnarray*}
Setting $R= r^{1/(1+(1/s -1/t)d)}$ this rearranges to 
$$r^{-d/(1+(1/s-1/t)d)} \int \phi_r^s(x-y) d\mu(y) \ \leq\ R^{-d}\int\phi_R^t(x-y) d\mu(y)
+\bigg(R^{-d}\int\phi_R^t(x-y) d\mu(y)\bigg)^{s/t}.$$

If $C_R^t (E) \geq R^{-d}$ for some $R$ then by Lemma \ref{equilib} there is an energy-minimising measure $\mu\in {\mathcal M}(E)$ such that the right-hand side of this inequality, and thus the left-hand side, is at most 2 for $\mu$-almost all $x$, so  
$C_r^s (E) \geq \frac{1}{2}r^{-d/((1+(1/s-1/t)d)}$ for the corresponding $r$. Letting $R\searrow 0$, t follows that
$$d(s)\ \geq\ \frac{d(t)}{1 +(1/s -1/t)d(t)},$$
where $d(\cdot)$ is either the lower or upper dimension profile, which rearranges to \eqref{ineq2}.
\end{proof}
\bigskip

Examples show that the inequalities in Lemma \ref{ineqs} give a complete characterisation of the dimension profiles that can be attained, see  \cite[Section 6]{FH2}. Note also that it follows easily from \eqref{ineq1} and \eqref{ineq2} that $d: \mathbb{R}^+ \to \mathbb{R}^+$ is continuous and, indeed, locally Lipschitz with constant 1.

\section{Packing dimensions}
\setcounter{equation}{0}
\setcounter{theo}{0}

In this final section we indicate how the results on box-counting dimensions carry over to the packing dimensions of projections and images of sets.

Packing measures and dimensions were introduced by Taylor and Tricot \cite{Tri,TT} as a type of dual to Hausdorff measures and dimensions, see \cite{Fa,Mat} for more recent expositions. Whilst, analogously to Hausdorff dimensions, packing dimensions can be defined by first setting up packing measures,  an equivalent definition in terms of  upper box dimensions  of countable coverings of a set is often more convenient in practice. For $E \subset \mathbb{R}^n$ we may define the {\it packing dimension} of $E$ by
\be\label{packdef}
\pkd E \ = \  \inf \Big\{\sup_{1\leq i<\infty} \ubd E_i : E \subset \bigcup_{i =1}^\infty E_i\Big\};
\ee
since the box dimension of a set equals that of its closure, we can assume that the sets $E_i$ in \eqref{packdef} are all compact.

It is natural to make an analogous definition of the {\it packing dimension profile} of $E \subset \mathbb{R}^n$ for $s>0$ by
\be\label{packpro}
\pkd^s E \ = \  \inf \Big\{\sup_{1\leq i<\infty} \ubd^s E_i : E \subset \bigcup_{i =1}^\infty E_i \mbox{ with each } E_i  \mbox{ compact} \Big\}.
\ee

By virtue of this definition, properties of packing dimension can be deduced from corresponding properties of upper box dimension. For example, we get an immediate analogue of Corollary \ref{sgeqn}.
\begin{cor}
Let $E\subset \mathbb{R}^n$. If $s\geq n$ then 
$$\pkd^s E\   =\   \pkd E.$$
\end{cor}

Similarly,  packing dimensions of H\"{o}lder images behave in the same way as box dimensions in Theorem \ref{liplem}.
\begin{cor}\label{packcor1}
Let $E\subset \mathbb{R}^n$ and let $f:E\to \mathbb{R}^m$ be an $\alpha$-H\"{o}lder map satisfying
$$
|f(x)-f(y)| \ \leq\  c |x-y|^\alpha \qquad (x,y \in E), 
$$
where $c>0$ and $0<\alpha \leq 1$. Then 
$$\pkd f( E)\  \leq \  \frac{1}{\alpha}\, \pkd^{m\alpha} E.$$
\end{cor}

\begin{proof}
If $t> \pkd^{m\alpha} E$ we may cover $E$ by a countable collection of sets $E_i$, which we may take to be compact, such that $\ubd^{m\alpha} E_i <t$.  By Theorem \ref{liplem}
$$\ubd f( E_i)\  \leq \  \frac{1}{\alpha}\, \ubd^{m\alpha} E_i \ \leq\  \frac{t}{\alpha}, $$
for all $i$, so the conclusion follows from  \eqref{packdef}.
\end{proof}

For  packing dimension bounds in the opposite direction we need the following property.

\begin{prop}\label{goodsubset}
Let $E\subset \mathbb{R}^n$ be a Borel set such that $\pkd^s E>t$. Then there exists a non-empty compact $F\subset E$ such that $\pkd^s (F\cap U )>t $ for every open set $U$ such that $F\cap U \neq \emptyset$.
\end{prop}

\noindent{\it Note on proof.}
In the special case where $E$ is compact there is a  short proof is based on \cite[Lemma 2.8.1]{BP}.
Let ${\mathcal B}$ be a countable basis of open sets that intersect $E$. Let 
$$ F\ =\ E\setminus \bigcup\big\{V \in {\mathcal B}: \pkd^s (E\cap V)\leq t\big\}.$$
Then $F$ is compact and, since  $ \pkd^s$ is countably stable, $ \pkd^s F >t$ and furthermore $ \pkd^s (E\setminus F) \leq t$.

Suppose for a contradiction that  $U$ is an open set such that $F\cap U \neq \emptyset$ and  $\pkd^s (F\cap U )\leq t $. As $ {\mathcal B}$ is a basis of open sets we may find  $V\subset U$ with $V \in {\mathcal B}$ such that $F\cap V \neq \emptyset$ and  $\pkd^s (F\cap V )\leq t$. Then
$$\pkd^s (E\cap V )\ \leq\ \max\{\pkd^s (E\setminus F ), \pkd^s (F\cap V )\} \ \leq t,$$
which contradicts that $V \in {\mathcal B}$.

For a general Borel set $E$ with $\pkd^s E>t$ we need to find a compact subset $E'\subset E$ with $\pkd^s E'>t$ which then has a suitable subset as above.
Whilst this is intuitively natural, I am not aware of a simple direct proof from the definition \eqref{packpro} of packing dimension profiles in terms of  box dimension profiles. The proof in \cite[Theorem 22]{How} uses a packing-type measure which relates to the $s$-packing dimension profile  in the same way that packing measure relates to packing dimension. This measure on $E$ is infinite, and by a `subset of finite measure' argument $E$ has a compact subset $E'$ of positive finite measure. \hfill$\Box$
\medskip

Assuming Proposition \ref{goodsubset} the packing dimension analogue of Theorem \ref{mainnew} follows easily.

\begin{cor}\label{packcor2}
Let $E\subset \mathbb{R}^n$ be a Borel set and let $s, \gamma>0$. Let $\{f_\omega: E \to \mathbb{R}^m,\ \omega \in \Omega\}$ be a family of continuous mappings such that for some $c>0$
$$\P(|\fw(x) - \fw(y)|\leq r)\ \leq\ c\,\phi_{r^\gamma}^s (x-y)
\qquad (x,y \in E, \ r>0).
$$
Then, for $\P$-almost all $\omega\in \Omega$, 
$$\pkd \fw( E)\  \geq \  \gamma\, \pkd^{s} E. $$
\end{cor}

\begin{proof}
Let $t< \pkd^s E$. By Proposition \ref{goodsubset} we may find a  non-empty compact $F\subset E$ such that for every open $U$ that intersects $F$,  $\pkd^s (F\cap U )>t $,  so in particular $\ubd^s (F\cap {\overline U} )>t $. As  $\mathbb{R}^n$ is seperable, there is a countable basis $\{U_i\}_{i=1}^\infty$ of open sets that intersect $F$.  By Theorem \ref{mainnew}, for $\P$-almost all $\omega\in \Omega$, 
\be\label{uibound}
 \ubd \fw(F\cap {\overline U}_i)\  \geq \   \gamma\, \ubd^{s}  (F\cap {\overline U}_i) >\gamma t
 \ee
for each $i$, and thus for all $i$ simultaneously.

For such an $\omega$, let $\{K_j\}_{j=1}^\infty$ be a cover of the compact set $\fw(F)$ by a countable collection of compact sets. By Baire's category theorem there is a $k$ and an open $V$ such that $\emptyset \neq \fw(F)\cap V\subset \fw(F)\cap K_k$. There is some $U_i$ such that $\fw(F \cap U_i)\subset \fw(F)\cap V$, so in particular 
$$\ubd (\fw(F)\cap K_k)\  \geq\ubd (\fw(F)\cap {\overline V})\  \geq \ \ubd \fw(F \cap {\overline U}_i)\  \geq \  \gamma t$$
by \eqref{uibound}. Since there is such a  $K_k$ for every cover of $\fw(F)$ by a countable collection of compact sets $\{K_j\}_{j=1}^\infty$, we conclude that $\pkd \fw(E)\geq \pkd \fw(F)\geq \gamma t$ by \eqref{packdef}.
\end{proof}

Corollaries \ref{packcor1} and \ref{packcor2} can be applied in exactly the same way as Theorems \ref{liplem} and \ref{mainnew} to  obtain, for example, packing dimension properties of projections and random images. We just state the basic projection result.  

\begin{theo}
Let $E\subset \mathbb{R}^n$ be Borel. Then  for all $V\in G(n,m)$
$$
\pkd \pi_V E \  \leq\   \pkd^m E,
$$
and for $\gamma_{n,m}$-almost all $V\in G(n,m)$
$$
\pkd \pi_V E \  \geq\   \pkd^m E.
$$
\end{theo}

\begin{proof}
The upper bound follows from Corollary \ref{packcor1} noting that $\pi_V: \mathbb{R}^n \to V$ is Lipschitz for all $V\in G(n,m)$.

As in Theorem \ref{mainnew}
$$\gamma_{n,m}\big\{V: |\pi_Vx-\pi_Vy| \leq r\big\}\ \leq\ a_{n,m}\,\phi_r^m (x-y)
\qquad (x\in \rn)$$
so taking $f_\omega$ as $\pi_V$ with $\gamma=1$ and $\P$ as $\gamma_{n,m}$ in Corollary \ref{packcor2} gives the almost sure lower bound.
\end{proof}

\bibliographystyle{plain}

\begin{thebibliography}{abc-20}

\bibitem{BP}
C. J. Bishop and Y. Peres.
{\em Fractals in Probability and Analysis},
Cambridge University Press, Cambridge, 2017.

\bibitem{Fa} K. J. Falconer. {\em Fractal Geometry - Mathemaical Foundations and Applications}. 3rd Ed., John Wiley, 2014.


\bibitem{FFJ} K. Falconer, J. Fraser and X. Jin X.  Sixty Years of Fractal Projections. In: {\em Fractal Geometry and Stochastics V} C. Bandt C, K. Falconer K. and M. Z\"{a}hle (eds),  Progress in Probability, {\bf 70}., Birkh\"{a}user, 2015.

\bibitem{FH}
K. J. Falconer and J. D. Howroyd. Projection theorems for box and packing dimensions, 
{\it Math. Proc. Cambridge Philos. Soc.} {\bf 119} (1996), 287--295.

\bibitem{FH2}
K. J. Falconer and J. D. Howroyd. Packing dimensions of projections
and dimension profiles, {\it Math. Proc. Cambridge Philos. Soc.} {\bf 121} (1997), 269--286.


\bibitem{FM}
K. J. Falconer and P. Mattila. The packing dimension of projections and sections of
measures, {\it Math. Proc. Cambridge Philos. Soc.} {\bf 119} (1996), 695--713.

\bibitem{How}
J. D. Howroyd. Box and packing dimensions of projections and dimension profiles, 
{\it Math. Proc. Cambridge Philos. Soc.} {\bf 130} (2001), 135--160.

\bibitem{Jar}
M. J\"{a}rvanp\"{a}\"{a}. On the upper Minkowski dimension, the packing dimension, and othogonal projections, {\it Ann. Acad. Sci. Fenn. A Dissertat.} {\bf 99} (1994).

\bibitem{Kah}
J.-P. Kahane.
{\em Some Random Series of Functions},
Cambridge University Press, Cambridge, 1985.


\bibitem{Kau}
R. Kaufman.
On Hausdorff dimension of projections,
 {\em Mathematika} {\bf 15} (1968), 153-155.
 
\bibitem{KM} 
 R. Kaufman and P. Mattila.
Hausdorff dimension and exceptional sets of linear transformations, 
{\em Ann. Acad. Sci. Fenn. Ser. A I Math.} {\bf 1} (1975), 387Ð392. 

\bibitem{Mar}
J. M. Marstrand.
 Some fundamental geometrical properties of plane sets of fractional
  dimensions, {\em Proc. London Math. Soc.(3)} {\bf 4} (1954), 257--302.

\bibitem{Mat4}
P.~Mattila.
 Hausdorff dimension, orthogonal projections and intersections with
  planes,
 {\em Ann. Acad. Sci. Fenn. A Math.} {\bf 1} (1975),  227--244.


\bibitem{Mat}
P. Mattila.
{\em Geometry of Sets and Measures in {E}uclidean Spaces},
Cambridge University Press, Cambridge, 1995.

\bibitem{Mat5}
P.~Mattila.
 Hausdorff dimension, projections, and the Fourier transform,
 {\em Publ. Mat.} {\bf 48} (2004),  3--48.

\bibitem{Mat2}
P. Mattila.
{\em Fourier Analysis and Hausdorff Dimension},
Cambridge University Press, Cambridge, 2015.

\bibitem{SX}
N.-R. Shieh and Y. Xiao.
Hausdorff and packing dimensions of the images of random fields, {\em Bernoulli} {\bf 16} (2010), 88--97.

\bibitem{Tri}
C. Tricot and S. J. Taylor.
Two definitions of fractional dimension, {\em Math. Proc. Cambridge Philos. Soc. } {\bf 91} (1982), 57--74.

\bibitem{TT}
C. Tricot.
Packing measure, and its evaluation for a Brownian path, {\em Trans. Amer. Math. Soc. }  {\bf 288} (1985), 679--699.

\bibitem{Xi}
Y. Xiao.
 Packing dimension of the image of fractional Brownian motion, {\em Statist. Probab. Lett.} {\bf 33} (1997), 379--387.




\end{thebibliography}

\bigskip
\end{document}